\documentclass[11pt]{article}

\usepackage[T1]{fontenc}
\usepackage{lmodern}
\usepackage{amsmath,amssymb,amsthm,mathtools}
\usepackage{enumitem}
\usepackage{microtype}
\usepackage[margin=1in]{geometry}
\usepackage[numbers,sort&compress]{natbib}
\usepackage[colorlinks=true,linkcolor=blue,citecolor=blue,urlcolor=blue]{hyperref}

\newtheorem{theorem}{Theorem}[section]
\newtheorem{proposition}[theorem]{Proposition}
\newtheorem{lemma}[theorem]{Lemma}
\newtheorem{corollary}[theorem]{Corollary}
\newtheorem{remark}[theorem]{Remark}

\newcommand{\R}{\mathbb R}
\newcommand{\E}{\mathbb E}
\newcommand{\PP}{\mathbb P}
\newcommand{\cP}{\mathcal P}
\newcommand{\cH}{\mathcal H}
\newcommand{\cD}{\mathcal D}
\newcommand{\eps}{\varepsilon}
\newcommand{\ind}{\mathbf 1}
\newcommand{\normaldist}{\mathrm{N}}

\begin{document}

\title{Two-Sample Inference for Gaussian-Smoothed Wasserstein Costs with Finite Moments}
\author{Jiaping Yang\thanks{School of Mathematical Sciences, Fudan University. Email: \texttt{jpyang22@m.fudan.edu.cn}}
\and Yunxin Zhang}
\date{}
\maketitle

\begin{abstract}
Gaussian smoothing has emerged as an effective technique for reducing the sample complexity of optimal transport. In this paper, we study the two-sample plug-in estimator of the Gaussian-smoothed Wasserstein cost \(T_p^{(\sigma)}(\mu,\nu)=W_p(\mu*\gamma_\sigma,\nu*\gamma_\sigma)^p\) on \(\R^d\).  For fixed smoothing and finite polynomial moments \(M_{q_\mu}(\mu)<\infty\), \(M_{q_\nu}(\nu)<\infty\), with \(q_\mu,q_\nu>p\), we establish upper bounds in probability of order \(\rho_{q_\mu,p,d}(m)+\rho_{q_\nu,p,d}(n)\).  Here \(\rho_{q,p,d}(N)=N^{-(q-p)/(q+d)}\) for \(p<q<d+2p\), \(N^{-1/2}\log N\) at \(q=d+2p\), and \(N^{-1/2}\) for \(q>d+2p\).  This order also holds in expectation under \(q_\mu,q_\nu\ge2p\).  When the smoothed population distance is positive, the cost bound yields this rate for the distance itself.  For \(p>1\) and \(q_\mu,q_\nu>d+2p\), we also derive a first-order expansion, a separated two-sample central limit theorem, and a sample-splitting variance estimator.
\end{abstract}

\noindent\textbf{Keywords:} Gaussian-smoothed Wasserstein distance; finite moments; two-sample inference; optimal transport; central limit theorem.

\noindent\textbf{MSC 2020:} 49Q22; 60F25; 60F05

\section{Introduction}
Wasserstein distances provide a flexible way to compare probability distributions in statistics, probability theory, and machine learning.  Their empirical behavior, however, is strongly affected by dimension.  Even when the underlying laws have moderate moments, empirical Wasserstein distances may converge slowly in high-dimensional settings.  This difficulty has motivated several regularized versions of empirical optimal transport.

One such regularization is to smooth the measures before computing the transport cost.  Convolution with a fixed Gaussian kernel replaces empirical measures by smooth, fully supported laws and gives more regular dual potentials.  These regularity properties reduce the complexity of the empirical processes that arise in plug-in estimation.  Since the Gaussian characteristic function has no zeros, the smoothing map is also injective; hence the smoothed Wasserstein distance still separates distinct probability laws.

For two probability measures \(\mu,\nu\) on \(\R^d\), write \(\mu^\sigma=\mu*\gamma_\sigma\) and \(\nu^\sigma=\nu*\gamma_\sigma\), where \(\gamma_\sigma\) denotes the centered Gaussian law with covariance \(\sigma^2I_d\).  We consider the Gaussian-smoothed cost with fixed bandwidth
\[
T_p^{(\sigma)}(\mu,\nu)=W_p(\mu^\sigma,\nu^\sigma)^p,
\]
where \(W_p\) denotes the \(p\)-Wasserstein distance.  Given independent samples \(X_1,\ldots,X_m\) from \(\mu\) and \(Y_1,\ldots,Y_n\) from \(\nu\), let \(\mu_m=m^{-1}\sum_{i=1}^m\delta_{X_i}\) and
\( \nu_n=n^{-1}\sum_{j=1}^n\delta_{Y_j}\). This paper studies the two-sample plug-in statistic \(T_p^{(\sigma)}(\mu_m,\nu_n)\).  We work primarily with the cost \(W_p^p\), rather than the distance \(W_p\), because the Kantorovich dual formula is linear in the two marginals and is therefore naturally suited to an analysis based on the two empirical processes.  Consequences for the smoothed distance itself are stated below.

For classical empirical Wasserstein distances, the object is usually \(W_p(\mu_N,\mu)\) or its \(p\)th power, and the assumptions range from compact support to finite polynomial moments.  The results in \citep{FournierGuillin2015,WeedBach2019,BobkovLedoux2019} show how the rate depends on the dimension, the transport order, and the available moment order.  These results are sharp in several regimes, but they also display the slow high-dimensional behavior of unsmoothed empirical optimal transport.

Gaussian smoothing changes the object rather than only the proof technique.  The smoothed optimal transport distance and its metric and statistical properties under fixed smoothing were introduced in \citep{GoldfeldGreenewald2020}.  Smoothed empirical measures for several probability metrics were studied in \citep{GoldfeldGreenewaldNilesWeedPolyanskiy2020}, while the smooth \(p\)-Wasserstein formulation, including empirical approximation and testing aspects, was developed in \citep{NietertGoldfeldKato2021}.  Subsequent work has refined different parts of this picture: dependent samples and sub-gamma tails \citep{ZhangChengReeves2021}, small-noise behavior \citep{DingNilesWeed2022}, sharp smooth-cost rates \citep{ManoleNilesWeed2024}, and fixed-smoothing limit distributions \citep{GoldfeldKatoNietertRioux2024}.  Inference for smooth \(1\)-Wasserstein and for regularized optimal transport is developed in \citep{SadhuGoldfeldKato2021,GoldfeldKatoRiouxSadhu2024}.  These papers identify the regularizing effect of smoothing, but their assumptions often involve compactness, sub-Gaussian tails, or empirical-process conditions stronger than finite polynomial moments.

A closely related recent contribution treats the one-sample discrepancy between a Gaussian-smoothed empirical measure and the corresponding smoothed population measure \citep{CossoMartiniPerelli2025}.  Its main results give mean convergence rates for Gaussian-smoothed Wasserstein distances under finite moment assumptions, with rate regimes closely related to the threshold \(q=d+2p\) used below.  The present paper differs in three concrete respects.  First, we work with the two-sample plug-in cost \(T_p^{(\sigma)}(\mu_m,\nu_n)\), which gives sharper distance consequences at separated alternatives. Second, the proof strategy is also different: rather than relying on one-sample metric inequalities and dyadic partitioning arguments, our analysis uses the Kantorovich dual formulation and entropy bounds for smoothed dual-potential classes.  Third, beyond rates, we prove a separated first-order expansion, a two-sample CLT, and a sample-splitting variance estimator. These conclusions are not direct consequences of existing smooth-Wasserstein limit theorems, which either treat one-sample approximation or impose stronger smooth empirical-process assumptions.

The first main result is a finite-moment rate bound.  If \(M_{q_\mu}(\mu)<\infty\) and \(M_{q_\nu}(\nu)<\infty\), with \(q_\mu,q_\nu>p\), then
\[
\bigl|T_p^{(\sigma)}(\mu_m,\nu_n)-T_p^{(\sigma)}(\mu,\nu)\bigr|
=O_{\PP}\{\rho_{q_\mu,p,d}(m)+\rho_{q_\nu,p,d}(n)\},
\]
where
\[
\rho_{q,p,d}(N)=
\begin{cases}
N^{-(q-p)/(q+d)},& p<q<d+2p,\\
N^{-1/2}\log N,& q=d+2p,\\
N^{-1/2},& q>d+2p .
\end{cases}
\]
This order also holds in expectation when \(q_\mu,q_\nu\ge2p\).  If the smoothed population distance is positive, the cost bound transfers to the distance without the \(1/p\)-power loss that would arise from applying the triangle inequality at the origin.  A simple testing consequence follows from this rate estimate.

The second main result concerns fixed alternatives.  When \(p>1\) and \(q_\mu,q_\nu>d+2p\), the shell entropy integral is finite.  Together with uniqueness and \(L_2\)-stability of normalized smoothed Kantorovich potentials, this yields a first-order expansion, a two-sample CLT for the smoothed cost, the corresponding distance CLT, and a sample-splitting variance estimator.  We do not address the null limit theorem, because the derivative at \(\mu=\nu\) is nonlinear and requires a different empirical-process analysis.

The final section records two extensions under sufficient conditions, to fixed smoothing kernels and to translation costs with polynomial-growth dual envelopes.  These extensions are included to indicate which parts of the rate proof are specific to the Gaussian kernel and the power cost.  The null case, \(p=1\), non-strictly convex costs, and vanishing smoothing remain outside the separated theory developed here.

The remainder of this paper is organized as follows.  Section~\ref{sec:bac} collects the dual formulation, growth bounds, and the shell bracketing estimate.  Section~\ref{sec:main} proves the rate and testing results.  Section~\ref{sec:sep} establishes the separated expansion, central limit theorem, and variance estimation.  Section~\ref{sec:extensions} discusses extensions and limitations.

\section{Background}
\label{sec:bac}
\subsection{Notation and dual preliminaries}
We adopt the following notation and conventions throughout the paper. For a probability law \(\mu\) on \(\R^d\), define
\[
M_r(\mu)=\int_{\R^d}(1+|x|^r)\mu(dx),
\qquad
\cP_r(\R^d)=\{\mu:M_r(\mu)<\infty\},
\]
and use the shorthand \(\mu f=\int f\,d\mu\).  For probability measures \(\mu,\nu\), let \(\Pi(\mu,\nu)\) denote the set of all couplings of \(\mu\) and \(\nu\).  When \(\mu,\nu\in\cP_p(\R^d)\), the \(p\)-Wasserstein distance is
\[
W_p(\mu,\nu)=
\left(\inf_{\pi\in\Pi(\mu,\nu)}
\int_{\R^d\times\R^d}|x-y|^p\,\pi(dx,dy)\right)^{1/p}.
\]
If \(\mu_N\) appears in a bound for empirical processes, it stands for the empirical measure of \(N\) independent observations drawn from \(\mu\).  Throughout, \(\gamma_\sigma=N(0,\sigma^2I_d)\) denotes the centered Gaussian law with covariance \(\sigma^2I_d\), \(\mu^\sigma=\mu*\gamma_\sigma\) denotes the Gaussian convolution of \(\mu\), and
\[
T_p^{(\sigma)}(\mu,\nu)=W_p(\mu^\sigma,\nu^\sigma)^p .
\]
We also write
\[
D_p^{(\sigma)}(\mu,\nu)=W_p(\mu^\sigma,\nu^\sigma)
\]
for the corresponding Gaussian-smoothed Wasserstein distance.
The convergence rate function is given by
\[
\rho_{q,p,d}(N)=
\begin{cases}
N^{-(q-p)/(q+d)},& p<q<d+2p,\\[0.2em]
N^{-1/2}\log N,& q=d+2p,\\[0.2em]
N^{-1/2},& q>d+2p.
\end{cases}
\]
We denote by \(C\) a generic positive constant that may vary from line to line but is always independent of the sample sizes. Subscripts, such as in \(C_M\) or \(C_\eta\), indicate dependence on specific parameters (e.g., \(M, \eta, \Delta_\sigma\)). Dependence on fixed parameters such as \(p,d,\sigma\) and moment bounds is suppressed unless it is explicitly discussed. All probability bounds hold in the sense of outer probability. The statistics \(T_p^{(\sigma)}(\mu_m,\nu_n)\) and \(W_p(\mu_m^\sigma,\nu_n^\sigma)\) themselves are measurable.

The Kantorovich dual formula for \(c_p(x,y)=|x-y|^p\) reads
\[
W_p(\mu,\nu)^p=
\sup_{\phi+\psi\le |x-y|^p}\left\{\int\phi\,d\mu+\int\psi\,d\nu\right\}.
\]
For a general cost \(c\), a feasible pair \((\phi,\psi)\) is called \(c\)-conjugate if
\[
\phi(x)=\inf_y\{c(x,y)-\psi(y)\},\qquad
\psi(y)=\inf_x\{c(x,y)-\phi(x)\}.
\]
For the power cost we write \(c_p(x,y)=|x-y|^p\).  Let \(\cD(\mu,\nu)\) denote a nonempty selected collection of normalized optimal dual pairs for \((\mu,\nu)\). In all bounds for empirical processes, these selections are taken from the representatives with polynomial growth guaranteed by Lemma~\ref{lem:growth}.  The precise additive normalization is immaterial, since empirical processes remain invariant under adding constants to the two potentials in opposite directions.

\begin{lemma}\label{lem:twosample-dual}
Let \(\mu,\nu,\widetilde\mu,\widetilde\nu\in\cP_p(\R^d)\).  Then
\[
\begin{aligned}
&\bigl|W_p(\widetilde\mu,\widetilde\nu)^p-W_p(\mu,\nu)^p\bigr| \\
&\quad\le
\sup_{(\phi,\psi)\in\cD(\mu,\nu)\cup\cD(\widetilde\mu,\widetilde\nu)}
\left|\int\phi\,d(\widetilde\mu-\mu)+\int\psi\,d(\widetilde\nu-\nu)\right| .
\end{aligned}
\]
With \((\widetilde\mu,\widetilde\nu)=(\mu_m^\sigma,\nu_n^\sigma)\) and \((\mu,\nu)=(\mu^\sigma,\nu^\sigma)\), this reduces the empirical smoothed cost error to the sum of two empirical processes over the convolved potential classes \(\phi*\gamma_\sigma\) and \(\psi*\gamma_\sigma\).
\end{lemma}

\begin{proof}
Let \((\widetilde\phi,\widetilde\psi)\in\cD(\widetilde\mu,\widetilde\nu)\).  Since \((\widetilde\phi,\widetilde\psi)\) is feasible for the pair \((\mu,\nu)\), we have
\[
W_p(\widetilde\mu,\widetilde\nu)^p-W_p(\mu,\nu)^p
\le \int\widetilde\phi\,d(\widetilde\mu-\mu)+\int\widetilde\psi\,d(\widetilde\nu-\nu).
\]
Conversely, if \((\phi,\psi)\in\cD(\mu,\nu)\), then feasibility for \((\widetilde\mu,\widetilde\nu)\) gives
\[
W_p(\mu,\nu)^p-W_p(\widetilde\mu,\widetilde\nu)^p
\le -\int\phi\,d(\widetilde\mu-\mu)-\int\psi\,d(\widetilde\nu-\nu).
\]
Taking absolute values and the supremum over the two sets of dual optimizers yields the first claim.  The smoothed form follows from \(\int\phi\,d\mu_m^\sigma=\int(\phi*\gamma_\sigma)\,d\mu_m\), and analogously for the other terms.
\end{proof}

\begin{lemma}\label{lem:growth}
Let \(p\ge1\) and \(\mu,\nu\in\cP_p(\R^d)\).  There exists an optimal dual pair \((\phi,\psi)\in\cD(\mu,\nu)\) satisfying
\[
|\phi(x)|+|\psi(y)|
\le C_p\{1+M_p(\mu)+M_p(\nu)+|x|^p+|y|^p\},
\qquad x,y\in\R^d .
\]
\end{lemma}

\begin{proof}
We use the following dominated-duality consequence from \citep[Theorem~5.10(ii)]{Villani2009}.  Let \(c\) be a lower semicontinuous finite cost satisfying
\[
c(x,y)\le a(x)+b(y),\quad \text{where} \quad a\in L_1(\mu),\quad b\in L_1(\nu),
\]
then the dual supremum is attained by a \(c\)-conjugate pair.  The construction in the cited theorem, followed by one additive shift \(u\mapsto u-t\), \(v\mapsto v+t\), gives normalized representatives satisfying
\[
|u(x)|\le a(x)+C_0,\qquad |v(y)|\le b(y)+C_0,
\]
where \(C_0\) is bounded by a constant multiple of
\[
1+\int a\,d\mu+\int b\,d\nu+\inf_{\pi\in\Pi(\mu,\nu)}\int c\,d\pi .
\]
This is the only property of the external duality theorem used below; the additive shift is harmless because the dual value and all empirical-process increments are invariant under opposite shifts of the two potentials.

For \(c_p(x,y)=|x-y|^p\), we may take
\[
a(x)=C_p(1+|x|^p),\qquad b(y)=C_p(1+|y|^p),
\]
since \(0\le c_p(x,y)\le a(x)+b(y)\) with \(a\in L_1(\mu)\) and \(b\in L_1(\nu)\).  In addition, 
\[
W_p(\mu,\nu)^p\le C_p\{M_p(\mu)+M_p(\nu)\}.
\]
Applying the dominated-duality bound yields an optimal pair \((\phi,\psi)\) satisfying
\[
|\phi(x)|+|\psi(y)|
\le C_p\{1+M_p(\mu)+M_p(\nu)+|x|^p+|y|^p\}.
\]
Henceforth, \(\cD(\mu,\nu)\) is always understood to be a nonempty selection of such normalized representatives.
\end{proof}

\begin{lemma}\label{lem:gaussian-smooth}
Let \(f\) be measurable and satisfy \(|f(y)|\le L(1+|y|^p)\) for some \(L > 0\). Then, for every integer \(k\ge0\),
\[
\max_{|a|\le k}|D^a(f*\gamma_\sigma)(x)|
\le C_{p,d,k,\sigma}L(1+|x|^p),
\qquad x\in\R^d .
\]
In particular, on the dyadic shell \(A_j=\{2^{j-1}<|x|\le2^j\}\), with \(A_0=B_1\) and \(r_j=2^j\vee1\), the preceding derivative bound reduces to \(C_{p,d,k,\sigma}Lr_j^p\).
\end{lemma}

\begin{proof}
For \(|a|\ge1\), we write \(D^a(f*\gamma_\sigma)(x)=\int f(y)D^a\varphi_\sigma(x-y)\,dy\), where \(\varphi_\sigma\) denotes the Gaussian density.  Since the derivative \(D^a\varphi_\sigma\) is a polynomial multiplied by a Gaussian density, the inequality \(|y|^p\le C_p(|x|^p+|x-y|^p)\) yields the asserted bound.  The case \(a=0\) is obtained by the analogous estimate without differentiating the kernel.
\end{proof}

\subsection{A shell bracketing estimate}
Let \(\mu\) be a probability measure with \(M_q(\mu)<\infty\) for some \(q>p\), and fix \(s=\lfloor d/2\rfloor+1.\) For a fixed constant \(L\ge1\), let \(\cH_L\) denote any class of \(C^s\) functions satisfying, on every dyadic shell \(A_j\),
\[
\max_{|a|\le s}\sup_{x\in A_j}|D^a h(x)|\le Lr_j^p .
\]
The \(C^s\)-entropy estimate on bounded subsets of \(\R^d\), combined with the shell decomposition, yields the following maximal inequality for deterministic function classes.  These classes may be replaced by pointwise measurable separable versions, or the bounds may be read in outer expectation.

\begin{lemma}\label{lem:shell-bound}
If \(M_q(\mu)<\infty\) for some \(q>p\), then
\[
\E\sup_{h\in\cH_L}|(\mu_N-\mu)h|
\le C L M_q(\mu)\rho_{q,p,d}(N).
\]
If in addition \(q\ge2p\), then
\[
\left(\E\left[\sup_{h\in\cH_L}|(\mu_N-\mu)h|^2\right]\right)^{1/2}
\le C L M_q(\mu)\rho_{q,p,d}(N).
\]
\end{lemma}

\begin{proof}
Let \(\mu_j=\mu(A_j)\).  The moment assumption yields \(\mu_j\le C M_q(\mu)r_j^{-q}\).  On \(A_j\), the class is contained in a \(C^s\) ball of smoothness \(s>d/2\), derivative radius \(Lr_j^p\), and diameter \(O(r_j)\).  Rescaling \(A_j\) to a fixed annulus multiplies derivatives of order \(|a|\) by at most \(r_j^{|a|}\).  The standard \(C^s\)-entropy bound on bounded sets therefore gives
\[
\log N(\delta,\cH_L|_{A_j},L_\infty(A_j))
\le C r_j^d\left(\frac{Lr_j^p}{\delta}\right)^{d/s}.
\]
An \(L_\infty(A_j)\)-bracket of width \(\delta\) induces an \(L_2(\mu)\)-bracket of width \(\delta \mu_j^{1/2}\).  Hence the \(L_2(\mu)\)-bracketing integral over the shell is obtained by substituting \(\delta=\eps \mu_j^{-1/2}\) in the preceding display.  Since \(d/s<2\), the power of \(\eps\) is integrable near zero, and
\[
J_{[]}(A_j)
:=\int_0^{Lr_j^p\mu_j^{1/2}}
\sqrt{1+\log N_{[]}(\eps,\cH_L\ind_{A_j},L_2(\mu))}\,d\eps
\le C L r_j^{p+d/2}\mu_j^{1/2}.
\]
For the shell class \(\mathcal F_j=\{h\ind_{A_j}:h\in\cH_L\}\), use the envelope \(F_j=CLr_j^p\ind_{A_j}\), whose \(L_2(\mu)\)-norm is \(CLr_j^p\mu_j^{1/2}\).  Applying the bracketing maximal inequality \citep[Theorem 2.14.2]{vdVW1996} to each \(\mathcal F_j\) and summing over \(j\le J\) (omitting shells with \(\mu_j=0\)), we obtain
\[
\E\sup_{h\in\cH_L}|(\mu_N-\mu)(h\ind_{B_R})|\le C\sum_{j=0}^J\left\{
N^{-1/2}J_{[]}(A_j)
{}+
N^{-1}\frac{J_{[]}(A_j)^2}{Lr_j^p\mu_j^{1/2}}
\right\}.
\]
The second term is the lower-order envelope contribution in the bracketing maximal inequality.  Substituting the bound on \(J_{[]}(A_j)\), it is at most
\[
C N^{-1}L\sum_{j=0}^J r_j^{p+d}\mu_j^{1/2}.
\]
The envelope on \(A_j\) is bounded by \(CLr_j^p\), so the tail outside \(B_R\) satisfies
\[
\E\sup_{h\in\cH_L}|(\mu_N-\mu)(h\ind_{B_R^c})|
\le CL\int_{|x|>R}(1+|x|^p)\mu(dx)
\le CLM_q(\mu)R^{p-q}.
\]
Combining the preceding estimates yields
\[
CLM_q(\mu)\left\{N^{-1/2}\sum_{j=0}^J r_j^{p+d/2-q/2}
 +N^{-1}\sum_{j=0}^J r_j^{p+d-q/2}
 +R^{p-q}\right\}.
\]
The three regimes follow by balancing the first shell sum with the moment tail.  If \(p<q<d+2p\), set \(R=N^{1/(q+d)}\). Then \(N^{-1/2}R^{p+d/2-q/2}\) and \(R^{p-q}\) are both of order \(N^{-(q-p)/(q+d)}\), while the \(N^{-1}\)-term is of smaller order.  If \(q=d+2p\), the first shell sum is logarithmic and this choice of \(R\) gives \(N^{-1/2}\log N\).  If \(q>d+2p\), the first shell sum is bounded.  The lower-order \(N^{-1}\)-term can then be controlled separately: if \(p+d-q/2<0\) we let \(R\to\infty\); if \(p+d-q/2=0\), take \(R=\exp(N^{1/4})\); and if \(p+d-q/2>0\), choose \(R=N^{1/\{2(p+d-q/2)\}}\).  In each subcase this term is \(O(N^{-1/2})\) or smaller, and the tail is also \(O(N^{-1/2})\).  This leads to the three regimes in the definition of \(\rho_{q,p,d}\).

For the second-moment bound, use the same truncation and define
\[
Z_N=\sup_{h\in\cH_L}|(\mu_N-\mu)h|,
\]
along with the envelope \(F(x)=CL(1+|x|^p)\).  Applying the square-function version of the same bracketing maximal inequality to the centered process indexed by the truncated class yields
\[
\left(\E\sup_{h\in\cH_L}|(\mu_N-\mu)(h\ind_{B_R})|^2\right)^{1/2}
\le C N^{-1/2}
\left\{\sum_{j=0}^J J_{[]}(A_j)+\|F\ind_{B_R}\|_{L_2(\mu)}\right\}.
\]
Here \(R=2^J\), and the term \(\|F\ind_{B_R}\|_{L_2(\mu)}\) is bounded by \(CLM_q(\mu)\) whenever \(q\ge2p\). For the tail, since \(|h|\le F\),
\[
\sup_{h\in\cH_L}|(\mu_N-\mu)(h\ind_{B_R^c})|
\le \mu_N(F\ind_{B_R^c})+\mu(F\ind_{B_R^c}).
\]
This gives
\[
\begin{aligned}
&\left(\E\sup_{h\in\cH_L}|(\mu_N-\mu)(h\ind_{B_R^c})|^2\right)^{1/2}\\
&\quad\le C\left\{\mu(F\ind_{B_R^c})
 +N^{-1/2}\|F\ind_{B_R^c}\|_{L_2(\mu)}\right\}\\
&\quad\le CLM_q(\mu)\left\{R^{p-q}+N^{-1/2}R^{(2p-q)/2}\right\},
\end{aligned}
\]
with the convention that the last factor \(R^{(2p-q)/2}\le1\) when \(q\ge2p\).  Combining the truncated and tail bounds yields
\[
(\E Z_N^2)^{1/2}
\le CLM_q(\mu)\left\{
N^{-1/2}\sum_{j=0}^J r_j^{p+d/2-q/2}
 +N^{-1/2}
 +R^{p-q}
 +N^{-1/2}R^{(2p-q)/2}
\right\}.
\]
If \(2p\le q<d+2p\), choose \(R=N^{1/(q+d)}\).  The first and third terms are of order \(N^{-(q-p)/(q+d)}\), while the remaining two terms are of no larger order. If \(q=d+2p\), choosing \(R=N^{1/(q+d)}\) yields \(N^{-1/2}\log N\).  If \(q>d+2p\), letting \(R\to\infty\) renders the shell sum finite, the tail vanishes, and the bound reduces to \(N^{-1/2}\). This proves the second-moment estimate.
\end{proof}

\section{Main rate results under finite moments}
\label{sec:main}
The proof has one technical point that is absent for deterministic marginals: the envelopes of the empirical dual potentials depend on empirical \(p\)-moments.  We first localize on events where these moments are bounded.  Lemmas~\ref{lem:growth} and~\ref{lem:gaussian-smooth} then place the convolved empirical potentials in a deterministic shell class \(\cH_L\).  The expectation bound is proved afterwards under the stronger square-integrability condition \(q\ge2p\).

\begin{theorem}[Two-sample rate in probability under finite moments]\label{thm:main-prob}
Let \(p\ge1\), \(\sigma>0\), and let \(X_1,\ldots,X_m\) and \(Y_1,\ldots,Y_n\) be independent samples with laws \(\mu\) and \(\nu\) on \(\R^d\).  Suppose that \(M_{q_\mu}(\mu)<\infty\) and \(M_{q_\nu}(\nu)<\infty\) for some \(q_\mu,q_\nu>p\).  Then
\[
\bigl|T_p^{(\sigma)}(\mu_m,\nu_n)-T_p^{(\sigma)}(\mu,\nu)\bigr|
=O_{\PP}\{\rho_{q_\mu,p,d}(m)+\rho_{q_\nu,p,d}(n)\}.
\]
Equivalently, for every \(\eta>0\) there is a constant \(C_\eta<\infty\) such that, for all sufficiently large \(m,n\),
\[
\PP\left(
\bigl|T_p^{(\sigma)}(\mu_m,\nu_n)-T_p^{(\sigma)}(\mu,\nu)\bigr|
> C_\eta\{\rho_{q_\mu,p,d}(m)+\rho_{q_\nu,p,d}(n)\}
\right)
\le\eta .
\]
\end{theorem}

\begin{proof}
By Lemma~\ref{lem:twosample-dual}, the cost error is bounded by the sum of the \(\mu\)- and \(\nu\)-empirical processes indexed by the smoothed potentials belonging to \(\cD(\mu^\sigma,\nu^\sigma)\cup\cD(\mu_m^\sigma,\nu_n^\sigma)\).

Fix \(M\ge1\) and consider the event \(E_M=\{M_p(\mu_m)+M_p(\nu_n)\le M\}\). Since \(q_\mu,q_\nu>p\), the sequences of empirical \(p\)-moments are tight, so for every \(\eta>0\) one can choose \(M\) large enough such that \(\limsup_{m,n}\PP(E_M^c)\le\eta/2\).  Gaussian smoothing also preserves the order of moments: for an independent \(Z\sim\gamma_\sigma\),
\[
M_p(\mu_m^\sigma)=\frac1m\sum_{i=1}^m\E_Z(1+|X_i+Z|^p)
\le C_{p,\sigma}\{1+M_p(\mu_m)\},
\]
and similarly \(M_p(\nu_n^\sigma)\le C_{p,\sigma}\{1+M_p(\nu_n)\}\). Hence, on \(E_M\), Lemma~\ref{lem:growth} provides an envelope with polynomial growth for the optimal potentials associated with the population pair \(\mu^\sigma,\nu^\sigma\) and the empirical pair \(\mu_m^\sigma,\nu_n^\sigma\). The deterministic constant depends on \(M\), \(M_p(\mu)\), \(M_p(\nu)\), \(p\), \(d\), and \(\sigma\), but not on \(m,n\) or on the realization inside \(E_M\).  Lemma~\ref{lem:gaussian-smooth} further shows that all relevant convolved potentials fall into the deterministic shell classes \(\cH_{L_M}^{\mu}\) and \(\cH_{L_M}^{\nu}\).

On \(E_M\), this gives
\[
\bigl|T_p^{(\sigma)}(\mu_m,\nu_n)-T_p^{(\sigma)}(\mu,\nu)\bigr|\le
\sup_{h\in\cH_{L_M}^{\mu}}| (\mu_m-\mu)h |+
\sup_{g\in\cH_{L_M}^{\nu}}| (\nu_n-\nu)g |.
\]
Lemma~\ref{lem:shell-bound} yields the expectation bounds 
\[
\E\sup_{h\in\cH_{L_M}^{\mu}}| (\mu_m-\mu)h |
\le C_M\rho_{q_\mu,p,d}(m),
\quad
\E\sup_{g\in\cH_{L_M}^{\nu}}| (\nu_n-\nu)g |
\le C_M\rho_{q_\nu,p,d}(n).
\]
By Markov's inequality, for a sufficiently large constant \(C_\eta\), the probability of the displayed error exceeding \(C_\eta\{\rho_{q_\mu,p,d}(m)+\rho_{q_\nu,p,d}(n)\}\) on \(E_M\) is at most \(\eta/2\).  Combining this with the tail bound \(\PP(E_M^c)\le\eta/2\) completes the proof.
\end{proof}

\begin{corollary}\label{cor:separated-distance}
Under the assumptions of Theorem~\ref{thm:main-prob}, recall the smoothed Wasserstein distance
\[
D_p^{(\sigma)}(\mu,\nu)=W_p(\mu^\sigma,\nu^\sigma),\qquad
\widehat D_{m,n}=D_p^{(\sigma)}(\mu_m,\nu_n),
\]
and write
\[
r_{m,n}=\rho_{q_\mu,p,d}(m)+\rho_{q_\nu,p,d}(n).
\]
If the smoothed population distance is separated from zero, namely
\[
\Delta_\sigma:=D_p^{(\sigma)}(\mu,\nu)>0,
\]
then
\[
|\widehat D_{m,n}-D_p^{(\sigma)}(\mu,\nu)|=O_{\PP}(r_{m,n}).
\]
\end{corollary}

\begin{proof}
Set \(A_{m,n}=\widehat D_{m,n}^p\) and \(A=\Delta_\sigma^p\).  Theorem~\ref{thm:main-prob} yields \( |A_{m,n}-A|=O_{\PP}(r_{m,n})\).  Since \(A>0\), the event \(\{|A_{m,n}-A|\le A/2 \} \) occurs with probability tending to one.  On this event, the map \(t\mapsto t^{1/p}\) is Lipschitz on \([A/2,\infty)\), whence
\[
|A_{m,n}^{1/p}-A^{1/p}|
\le \frac1p(A/2)^{1/p-1}|A_{m,n}-A|.
\]
This establishes the distance rate under separation.  For the Gaussian kernel, \(\Delta_\sigma>0\) is equivalent to \(\mu\ne\nu\), since the Gaussian characteristic function is non-vanishing and convolution by \(\gamma_\sigma\) is injective. 
\end{proof}

\begin{remark}
 For comparison, the triangle inequality yields
\[
|\widehat D_{m,n}-D_p^{(\sigma)}(\mu,\nu)|
\le D_p^{(\sigma)}(\mu_m,\mu)+D_p^{(\sigma)}(\nu_n,\nu).
\]
Applying the one-sample version of Theorem~\ref{thm:main-prob} to each term gives one-sample cost errors of orders \(\rho_{q_\mu,p,d}(m)\) and \(\rho_{q_\nu,p,d}(n)\), respectively. Taking \(p\)th roots merely produces 
\[
|\widehat D_{m,n}-D_p^{(\sigma)}(\mu,\nu)|
=O_{\PP}\{\rho_{q_\mu,p,d}(m)^{1/p}+\rho_{q_\nu,p,d}(n)^{1/p}\}.
\]
For \(p>1\) and separated alternatives, the direct two-sample cost bound avoids the usual loss caused by taking the \(p\)th root at the origin.
\end{remark}

\begin{theorem}[Expectation rate under square-integrable envelopes]\label{thm:mean}
Under the assumptions of Theorem~\ref{thm:main-prob}, suppose in addition that \(q_\mu\ge2p\) and \(q_\nu\ge2p\).  Then
\[
\E\bigl|T_p^{(\sigma)}(\mu_m,\nu_n)-T_p^{(\sigma)}(\mu,\nu)\bigr|
\le C\{\rho_{q_\mu,p,d}(m)+\rho_{q_\nu,p,d}(n)\}.
\]
The distance satisfies
\[
\E\left|W_p(\mu_m^\sigma,\nu_n^\sigma)-W_p(\mu^\sigma,\nu^\sigma)\right|
\le C\{\rho_{q_\mu,p,d}(m)+\rho_{q_\nu,p,d}(n)\}^{1/p}.
\]
\end{theorem}

\begin{proof}
The argument combines the localization scheme developed in Theorem~\ref{thm:main-prob} with the square-integrable envelope structure. Let
\[
L_{m,n}=C\{1+M_p(\mu_m)+M_p(\nu_n)+M_p(\mu)+M_p(\nu)\}
\]
be the random envelope radius produced by Lemmas~\ref{lem:growth} and~\ref{lem:gaussian-smooth}.  Let \(\cH_1^\mu\) and \(\cH_1^\nu\) denote the corresponding unit shell classes, i.e., the classes satisfying the derivative bounds of \(\cH_L\) with \(L=1\).  On every realization, each relevant convolved potential lies in \(L_{m,n}\cH_1^\mu\) on the \(\mu\)-side and in \(L_{m,n}\cH_1^\nu\) on the \(\nu\)-side.  The dual reduction yields
\[
\bigl|T_p^{(\sigma)}(\mu_m,\nu_n)-T_p^{(\sigma)}(\mu,\nu)\bigr|
\le L_{m,n}(Z_m^\mu+Z_n^\nu),
\]
where
\[
Z_m^\mu=\sup_{h\in\cH_1^\mu}|(\mu_m-\mu)h|,
\qquad
Z_n^\nu=\sup_{g\in\cH_1^\nu}|(\nu_n-\nu)g|.
\]
Since \(q_\mu,q_\nu\ge2p\), the empirical \(p\)-moment envelopes are square integrable:
\[
\sup_m \E M_p(\mu_m)^2<\infty,\qquad
\sup_n \E M_p(\nu_n)^2<\infty.
\]
Indeed, \(M_p(\mu_m)\le m^{-1}\sum_{i=1}^m(1+|X_i|^p)\), and the right-hand side possesses a uniformly bounded second moment whenever \(M_{2p}(\mu)<\infty\); the \(\nu_n\) term is handled identically.
It follows that \(\sup_{m,n}\E L_{m,n}^2<\infty\).  By the Cauchy--Schwarz inequality and the second-moment part of Lemma~\ref{lem:shell-bound},
\[
\E L_{m,n}Z_m^\mu
\le (\E L_{m,n}^2)^{1/2}(\E (Z_m^\mu)^2)^{1/2}
\le C\rho_{q_\mu,p,d}(m),
\]
and the \(\nu\)-side gives \(\E L_{m,n}Z_n^\nu\le C\rho_{q_\nu,p,d}(n)\).  This proves the cost bound.  The condition \(q\ge2p\) is used only to make the random polynomial envelope square integrable. The rate itself is still determined by the threshold arising from the shell decomposition, namely \(q=d+2p\).

The distance bound follows from the elementary inequality \(|a-b|\le |a^p-b^p|^{1/p}\) (\(a,b>0\)), combined with an application of Jensen's inequality.    
\end{proof}

\begin{corollary}\label{cor:mean-separated-distance}
Under the assumptions of Theorem~\ref{thm:mean}, recall
\[
D_p^{(\sigma)}(\mu,\nu)=W_p(\mu^\sigma,\nu^\sigma),\qquad
\Delta_\sigma=D_p^{(\sigma)}(\mu,\nu),\qquad
\widehat D_{m,n}=D_p^{(\sigma)}(\mu_m,\nu_n).
\]
If \(\Delta_\sigma>0\), then
\[
\E\left|\widehat D_{m,n}-\Delta_\sigma\right|
\le C_\Delta\{\rho_{q_\mu,p,d}(m)+\rho_{q_\nu,p,d}(n)\}.
\]
For the Gaussian kernel, \(\mu\ne\nu\) implies \(\Delta_\sigma>0\), so the conclusion applies at every fixed alternative.
\end{corollary}

\begin{proof}
Let \(A=\Delta_\sigma^p>0\) and \(A_{m,n}=\widehat D_{m,n}^p\).  The function
\[
t\longmapsto \frac{|t^{1/p}-A^{1/p}|}{|t-A|},\qquad t\ge0,
\]
with its continuous extension at \(t=A\), is bounded since \(A>0\).  Hence there exists \(C_\Delta<\infty\) such that
\[
|t^{1/p}-A^{1/p}|\le C_\Delta|t-A|,
\qquad t\ge0.
\]
Applying this inequality with \(t=A_{m,n}\) and invoking Theorem~\ref{thm:mean} yields
\[
\E|\widehat D_{m,n}-\Delta_\sigma|
\le C_\Delta\E|A_{m,n}-A|
\le C_\Delta\{\rho_{q_\mu,p,d}(m)+\rho_{q_\nu,p,d}(n)\}.
\]
If \(\mu^\sigma=\nu^\sigma\), their characteristic functions coincide after multiplication by the non-vanishing Gaussian characteristic function, implying \(\mu=\nu\).  Hence \(\mu\ne\nu\) entails \(\Delta_\sigma>0\).
\end{proof}
\begin{corollary}\label{cor:testing}
 Under the assumptions of Theorem~\ref{thm:main-prob}, set
\[
r_{m,n}=\rho_{q_\mu,p,d}(m)+\rho_{q_\nu,p,d}(n).
\]
For any threshold sequence \(\theta_{m,n}\) satisfying \(\theta_{m,n}/r_{m,n}\to\infty\), the test \(\varphi_{m,n}=\ind\{\widehat D_{m,n}^p>\theta_{m,n}\}\) has vanishing rejection probability under \(\mu=\nu\).  Furthermore, for any fixed constants \(K_\mu,K_\nu<\infty\), the power tends to one uniformly over alternatives satisfying
\[
M_{q_\mu}(\mu)\le K_\mu,\qquad M_{q_\nu}(\nu)\le K_\nu,\qquad
D_p^{(\sigma)}(\mu,\nu)^p\ge 2\theta_{m,n}.
\]
In particular, if \(\theta_{m,n}\to0\), then the power tends to one at every fixed alternative with \(D_p^{(\sigma)}(\mu,\nu)>0\) and finite \(q_\mu\)- and \(q_\nu\)-moments.
\end{corollary}

\begin{proof}
The null statement follows from Theorem~\ref{thm:main-prob} because \(D_p^{(\sigma)}(\mu,\mu)^p=0\).  Under the displayed alternatives, the triangle inequality applied to the cost error yields
\[
\widehat D_{m,n}^p
\ge D_p^{(\sigma)}(\mu,\nu)^p
-
\bigl|T_p^{(\sigma)}(\mu_m,\nu_n)-T_p^{(\sigma)}(\mu,\nu)\bigr|,
\]
and the error is \(o_{\PP}(\theta_{m,n})\), uniformly over the displayed moment balls.  Indeed, the proof of Theorem~\ref{thm:main-prob} is uniform when \(M_{q_\mu}(\mu)\le K_\mu\) and \(M_{q_\nu}(\nu)\le K_\nu\): the deterministic shell bounds in Lemma~\ref{lem:shell-bound} depend on the laws only through these moment bounds, and the localization event satisfies
\[
\sup_{M_{q_\mu}(\mu)\le K_\mu}
\PP\{M_p(\mu_m)>M\}
\le M^{-1}\sup_{M_{q_\mu}(\mu)\le K_\mu}M_p(\mu)
\le C M^{-1}K_\mu,
\]
with an identical estimate on the \(\nu\)-side.  These bounds make the constants in Theorem~\ref{thm:main-prob} uniform over the alternatives in the corollary.
The final statement for fixed alternatives follows because \(D_p^{(\sigma)}(\mu,\nu)^p>0\) and \(\theta_{m,n}\to0\) imply \(D_p^{(\sigma)}(\mu,\nu)^p\ge2\theta_{m,n}\) eventually, while the finite moments place the fixed alternative in some moment ball.
\end{proof}

\begin{remark}[Dependence on the smoothing bandwidth]\label{rem:sigma-dependence}
The constants in the preceding rate bounds are stated for fixed \(\sigma>0\).  Their dependence on \(\sigma\) can be read from the proof. Let \(s=\lfloor d/2\rfloor+1\) be the differentiability order used in the shell bracketing argument.  If the laws are restricted to moment balls
\[
M_{q_\mu}(\mu)\le K_\mu,\qquad M_{q_\nu}(\nu)\le K_\nu,
\]
then there exists a constant \(C<\infty\), depending only on \(p,d,q_\mu,q_\nu,K_\mu,K_\nu\) and on the confidence level in probability statements, such that the constants in the cost bounds of Theorems~\ref{thm:main-prob} and~\ref{thm:mean} may be chosen no larger than
\[
C(1+\sigma^p)^2(1+\sigma^{-s}).
\]
Indeed, smoothing changes \(p\)-moments by at most a factor of order \(1+\sigma^p\):
\[
M_p(\mu^\sigma)\le C_p(1+\sigma^p)M_p(\mu),
\]
after adjusting the constant already included in the definition of \(M_p(\mu)\).  Lemma~\ref{lem:growth} then gives order-\(p\) envelopes for the relevant Kantorovich potentials with this same polynomial dependence on \(\sigma\).  The additional factor comes from Gaussian differentiation:
\[
\int_{\R^d}(1+|z|^p)|D^a\varphi_\sigma(z)|\,dz
\le C_{p,d,a}(1+\sigma^p)\sigma^{-|a|},
\qquad |a|\le s,
\]
where \(\varphi_\sigma\) is the \(N(0,\sigma^2I_d)\) density.  These two estimates yield the displayed bound for the radius of the shell class, and the maximal inequalities are linear in that radius.

For the distance bounds under separation, this cost constant is multiplied by the local Lipschitz constant of \(t\mapsto t^{1/p}\) at \(T_p^{(\sigma)}(\mu,\nu)=\Delta_\sigma^p\), which is of order \(\Delta_\sigma^{1-p}\) for \(p>1\).  No statement uniform in \(\sigma\) is implied unless the smoothed separation \(\Delta_\sigma\) is bounded away from zero.  The bound above is not optimized; it records only the dependence on the fixed bandwidth used by the proof and explains why a regime with vanishing smoothing requires a separate bias-variance analysis.
\end{remark}

\section{Separated inference under high moment assumptions}
\label{sec:sep}
The preceding rate theorem gives only the size of the two-sample error.  At separated alternatives a sharper statement is available when the moment order is beyond the shell threshold.  The condition \(q>d+2p\) makes the shell entropy integral finite.  This yields a Donsker theorem for the localized smoothed dual classes.  Combined with uniqueness and \(L_2\)-stability of the empirical smoothed potentials, it yields the first-order expansion needed for a central limit theorem.

We keep the restriction \(p>1\).  For \(p=1\), optimal Lipschitz potentials need not be unique, even after smoothing, and the limit theory is typically expressed through a set of optimal potentials; see, for instance, \citep{SadhuGoldfeldKato2021}.  Directional delta method arguments for non-smooth optimal transport statistics are discussed in \citep{SommerfeldMunk2018,FangSantos2019}.  The result below is complementary to the full smooth Wasserstein limit theory in \citep{GoldfeldKatoNietertRioux2024} and is specialized to fixed separated alternatives and scalar CLTs under the moment condition \(q_\mu,q_\nu>d+2p\).

\subsection{Potential uniqueness and stability}
The limit theory uses two external facts about Kantorovich potentials.  They are not new results of this paper, but we state the particular forms needed here so that the subsequent verification for Gaussian-smoothed laws is explicit.

The first input is the connected-support uniqueness result from \citep[Theorem~2]{StaudtHundrieserMunk2025}.

\begin{lemma}[Connected-support uniqueness]\label{lem:source-connected-uniqueness}
 Let \(\mu,\nu\in\cP_p(\R^d)\), and let \(c(x,y)=h(x-y)\), where \(h\) is locally Lipschitz, convex, superlinear, and differentiable.  Suppose that the interior of \(\operatorname{supp}(\mu)\) is connected, the boundary of \(\operatorname{supp}(\mu)\) is negligible for \(\mu\), and the induced-irregularity set of the cost is contained in this boundary.  Then first Kantorovich potentials for the cost \(c\) from \(\mu\) to \(\nu\) are unique \(\mu\)-a.s. up to additive constants.  If the corresponding \(c\)-conjugate pairs are used, the second potentials are unique \(\nu\)-a.s. up to the opposite additive constants.
\end{lemma}

\begin{proposition}[Uniqueness of normalized potentials after smoothing]\label{prop:unique-potentials}
Let \(p>1\), \(\sigma>0\), and \(\mu,\nu\in\cP_p(\R^d)\).  Then \(\mu^\sigma\) and \(\nu^\sigma\) have strictly positive \(C^\infty\) densities on \(\R^d\).  For the cost \(c_p(x,y)=|x-y|^p\), the Kantorovich potentials for \(W_p(\mu^\sigma,\nu^\sigma)^p \) are unique up to additive constants.  After any fixed normalization, this determines a unique potential pair \((\phi_{\sigma},\psi_{\sigma})\) and unique convolved dual functions
\[
h_{\sigma}(x)=\int \phi_{\sigma}(x+z)\,\gamma_\sigma(dz),
\qquad
k_{\sigma}(y)=\int \psi_{\sigma}(y+z)\,\gamma_\sigma(dz),
\]
up to the corresponding centering constants.
\end{proposition}

\begin{proof}
The Gaussian convolution yields strictly positive \(C^\infty\) densities on \(\R^d\).  Hence both smoothed laws have support \(\R^d\), whose interior is connected and whose boundary is empty.

We apply Lemma~\ref{lem:source-connected-uniqueness}.  The required regularity condition is verified by the Euclidean interior regularity result for costs \(c(x,y)=h(x-y)\) in \citep{StaudtHundrieserMunk2025}: for \(h(z)=|z|^p\), \(p>1\), the function \(h\) is locally Lipschitz, convex, superlinear, and differentiable, so the set of induced irregularity is contained in the boundary of the source support.  Since that boundary is empty here, the connected support theorem applies to the pair \((\mu^\sigma,\nu^\sigma)\).  The first Kantorovich potentials are unique \(\mu^\sigma\)-a.s. up to additive constants.  Because \(\mu^\sigma\) has a strictly positive density, this is equivalent to Lebesgue-a.e. uniqueness on \(\R^d\).  The potentials may be chosen \(c_p\)-conjugate.  Such representatives are locally Lipschitz on \(\R^d\) for the power cost by the local regularity cited above.  Equality Lebesgue-a.e. of two normalized first potentials then implies equality everywhere, since their difference is continuous.  The \(c_p\)-conjugacy relation identifies the second potential everywhere up to the opposite constant:
\[
\psi(y)=\inf_x\{|x-y|^p-\phi(x)\}.
\]
Convolution preserves additive constants and depends only on the Lebesgue-a.e. value of the integrand against the Gaussian density. Hence the centered convolved dual functions are unique.
\end{proof}

Let \((\phi_\sigma,\psi_\sigma)\) be an optimal pair for \((\mu^\sigma,\nu^\sigma)\).  We choose its additive normalization so that
\[
\mu h_\sigma=0,
\quad \text{where} \quad
h_\sigma(x)=\int\phi_\sigma(x+z)\,\gamma_\sigma(dz),
\qquad
k_\sigma(y)=\int\psi_\sigma(y+z)\,\gamma_\sigma(dz).
\]
The influence functions are then
\[
\bar h_\sigma=h_\sigma,
\qquad
\bar k_\sigma=k_\sigma-\nu k_\sigma .
\]
Adding constants to a feasible dual pair does not change either the dual objective or the increments of the empirical processes, so this normalization is only a device for selecting one representative.  By Lemmas~\ref{lem:growth} and~\ref{lem:gaussian-smooth}, \(h_\sigma,k_\sigma,\bar h_\sigma\), and \(\bar k_\sigma\) have polynomial envelope of order \(p\).

For \(L\ge1\), let \(\mathcal A_L\) denote the collection of normalized pairs \((h,k)\) obtained by convolving feasible \(c_p\)-conjugate pairs \((\phi,\psi)\) satisfying the polynomial envelope in Lemma~\ref{lem:growth} with constant at most \(L\), and then shifting the pair so that \(\mu h=0\).  The fixed normalization removes the ambiguity from additive constants and places the class in \(L_2(\mu)\oplus L_2(\nu)\); the value of \(\mu h+\nu k\) and of \(\mu_m h+\nu_n k\) is unchanged by this shift.

\begin{lemma}[Donsker property of localized smoothed dual classes]\label{lem:localized-donsker}
Let \(p\ge1\), \(\sigma>0\), and assume \(M_q(\mu)<\infty\) for some \(q>d+2p\).  For every fixed \(L\), any class generated in one coordinate as above by convolved potentials with envelope constant at most \(L\) is \(\mu\)-Donsker, meaning that the empirical process indexed by this class converges weakly to a tight Gaussian process in \(\ell^\infty\) of the class, and is totally bounded in \(L_2(\mu)\).  The class of pairs \(\mathcal A_L\) is Donsker and totally bounded in \(L_2(\mu)\oplus L_2(\nu)\) whenever \(M_{q_\mu}(\mu)+M_{q_\nu}(\nu)<\infty\) with \(q_\mu,q_\nu>d+2p\).
\end{lemma}

\begin{proof}
We work with pointwise measurable separable versions of the classes; equivalently, all statements about empirical processes may be read in outer probability.  By Lemmas~\ref{lem:growth} and~\ref{lem:gaussian-smooth}, every class in one coordinate is contained in a shell class \(\cH_{C L}\) with envelope \(F(x)=CL(1+|x|^p)\).  Since \(q>d+2p\), we have \(F\in L_2(\mu)\).

On each shell \(A_j\), the calculation in Lemma~\ref{lem:shell-bound} yields the bracketing integral
\[
J_{[]}(A_j)\le C L r_j^{p+d/2}\mu(A_j)^{1/2}.
\]
For finitely many shells \(A_0,\ldots,A_J\), brackets on the shells can be added: if \([l_j,u_j]\) brackets \(h\ind_{A_j}\), then
\[
\left[\sum_{j=0}^J l_j,\sum_{j=0}^J u_j\right]
\]
brackets \(h\ind_{B_{2^J}}\), and its \(L_2(\mu)\)-width is at most \((\sum_j\|u_j-l_j\|_{\mu,2}^2)^{1/2}\).  Allocating the shell widths so that the squared widths sum to the desired global width gives
\[
J_{[]}(\|F\ind_{B_{2^J}}\|_{\mu,2},\cH_{CL}\ind_{B_{2^J}},L_2(\mu))
\le C\sum_{j=0}^J J_{[]}(A_j).
\]
The moment bound \(\mu(A_j)\le C M_q(\mu)r_j^{-q}\) together with \(q>d+2p\) implies 
\[
\sum_{j\ge0}J_{[]}(A_j)
\le C L M_q(\mu)^{1/2}
\sum_{j\ge0}r_j^{p+d/2-q/2}<\infty .
\]
Since \(F\in L_2(\mu)\), choose \(J\) large enough that
\[
\|F\ind_{B_{2^J}^c}\|_{L_2(\mu)}<\varepsilon .
\]
The single tail bracket \([-F\ind_{B_{2^J}^c},F\ind_{B_{2^J}^c}]\) then has \(L_2(\mu)\)-width less than \(2\varepsilon\), while the truncated class has finite bracketing integral uniformly in \(J\).  Letting \(\varepsilon\downarrow0\) yields a finite global bracketing integral for the full class.

The bracketing central limit theorem for empirical processes \citep[Theorem~2.5.6]{vdVW1996} yields the \(\mu\)-Donsker property.  Finite bracketing numbers at every positive \(L_2(\mu)\)-radius also entail total boundedness in \(L_2(\mu)\).  The statement for pairs follows by applying the result for one coordinate to the direct sum metric
\[
\|(h,k)\|_{L_2(\mu)\oplus L_2(\nu)}
=\{\|h\|_{L_2(\mu)}^2+\|k\|_{L_2(\nu)}^2\}^{1/2}
\]
and using the independence of the two empirical processes.
\end{proof}

Define
\[
\mathbb M(h,k)=\mu h+\nu k,
\qquad
\mathbb M_{m,n}(h,k)=\mu_m h+\nu_n k.
\]

\begin{lemma}[Localized equicontinuity for random dual pairs]\label{lem:localized-equicont}
Assume \(m/(m+n)\to\lambda\in(0,1)\), put \(N_{m,n}=mn/(m+n)\), and let
\[
d_2\{(h,k),(h',k')\}
=\|(h-h',k-k')\|_{L_2(\mu)\oplus L_2(\nu)} .
\]
Let \((H_{m,n},K_{m,n})\) be random pairs and let \((h_0,k_0)\) be fixed.  Assume that, for every \(\eta>0\), there exists \(L<\infty\) such that, for all sufficiently large \(m,n\),
\[
\PP\{(H_{m,n},K_{m,n})\in\mathcal A_L,\ (h_0,k_0)\in\mathcal A_L\}\ge1-\eta,
\]
and suppose that
\[
d_2\{(H_{m,n},K_{m,n}),(h_0,k_0)\}\to_{\PP}0 .
\]
Then
\[
N_{m,n}^{1/2}\Bigl[(\mathbb M_{m,n}-\mathbb M)(H_{m,n},K_{m,n})
-(\mathbb M_{m,n}-\mathbb M)(h_0,k_0)\Bigr]=o_{\PP}(1).
\]
\end{lemma}

\begin{proof}
For fixed \(L\), Lemma~\ref{lem:localized-donsker} gives asymptotic equicontinuity of the one-sample empirical processes on the two coordinate classes.  Since
\[
N_{m,n}^{1/2}(\mathbb M_{m,n}-\mathbb M)(h,k)
=\left(\frac{N_{m,n}}{m}\right)^{1/2}\mathbb G_m^\mu h
+\left(\frac{N_{m,n}}{n}\right)^{1/2}\mathbb G_n^\nu k,
\]
where \(\mathbb G_m^\mu=m^{1/2}(\mu_m-\mu)\) and \(\mathbb G_n^\nu=n^{1/2}(\nu_n-\nu)\), and \(N_{m,n}/m\to1-\lambda\), \(N_{m,n}/n\to\lambda\), the two-sample process is asymptotically equicontinuous on \(\mathcal A_L\) under the metric \(d_2\).  Hence, for every \(\varepsilon>0\),
\[
\lim_{\delta\downarrow0}\limsup_{m,n}
\PP\left\{
\sup_{\substack{a,b\in\mathcal A_L\\ d_2(a,b)<\delta}}
N_{m,n}^{1/2}|(\mathbb M_{m,n}-\mathbb M)(a)-(\mathbb M_{m,n}-\mathbb M)(b)|>\varepsilon
\right\}=0 .
\]
Intersecting this event with the localization event above, whose probability is high, and using the assumed \(d_2\)-convergence proves the claim. Letting \(\eta\downarrow0\) completes the proof.
\end{proof}

The second input controls how optimal potentials move when both marginals vary.  This is the step that allows the random empirical maximizer to be replaced by the fixed population potential in the linear expansion.  We use the following local stability theorem for potentials \citep[Theorem~3.4]{delBarrioGonzalezSanzLoubes2024}.

\begin{lemma}[Local stability of potentials]\label{lem:source-potential-stability}
Let \(c(x,y)=h(x-y)\), where \(h:\R^d\to[0,\infty)\) is differentiable and satisfies the following assumptions: \(h\) is strictly convex; for each height \(r>0\) and angle \(\theta\in(0,\pi)\), sufficiently far from the origin there is a cone of height \(r\) and angle \(\theta\) on which \(h\) attains its maximum at the vertex; and \(h(z)/|z|\to\infty\) as \(|z|\to\infty\).  Let \(\mu_n,\nu_n,\mu,\nu\) be probability laws on \(\R^d\) such that \(\mu_n\Rightarrow \mu\), \(\nu_n\Rightarrow \nu\), \(T_c(\mu_n,\nu_n)<\infty\), and \(T_c(\mu,\nu)<\infty\).  Suppose that \(\mu\) is absolutely continuous with respect to Lebesgue measure and has connected support with negligible boundary.  If \((\phi_n,\psi_n)\) and \((\phi,\psi)\) are \(c\)-conjugate optimal pairs for \((\mu_n,\nu_n)\) and \((\mu,\nu)\), respectively, then there are constants \(a_n\in\R\) such that
\[
\phi_n-a_n\to\phi
\]
locally uniformly on compact subsets of \(\operatorname{supp}(\mu)\).  The analogous assertion holds for the second coordinate after interchanging the two marginals, provided the corresponding limiting source law also has absolutely continuous connected support with negligible boundary.
\end{lemma}

\begin{lemma}[\(L_2\)-stability of empirical smoothed potentials]\label{lem:l2-potential-stability}
Let \(p>1\), \(\sigma>0\), and suppose \(M_{q_\mu}(\mu)+M_{q_\nu}(\nu)<\infty \) for some \(q_\mu,q_\nu>d+2p\).  Let \((\widetilde \phi_{m,n},\widetilde \psi_{m,n})\) be any \(c_p\)-conjugate optimal dual pair for \((\mu_m^\sigma,\nu_n^\sigma)\), chosen from the representatives with polynomial growth in Lemma~\ref{lem:growth}, and let \((\widehat \phi_{m,n},\widehat \psi_{m,n})\) be the shift of this pair by opposite constants satisfying
\[
\mu \widehat h_{m,n}=0,\qquad
\widehat h_{m,n}=\widehat\phi_{m,n}*\gamma_\sigma,\qquad
\widehat k_{m,n}=\widehat\psi_{m,n}*\gamma_\sigma .
\]
Then, as \(m,n\to\infty\),
\[
\|(\widehat h_{m,n},\widehat k_{m,n})-(h_\sigma,k_\sigma)\|_{L_2(\mu)\oplus L_2(\nu)}
\to_{\PP}0 .
\]
If the selected optimal pair is not known to be measurable, the convergence is understood in outer probability.
\end{lemma}

\begin{proof}
The argument is pathwise on an event of probability one; if the selected empirical pair is not measurable, the same conclusion is read in outer probability.
The two-index notation is justified by applying the sequential stability statement along arbitrary deterministic sequences \(m_\ell,n_\ell\to\infty\); the resulting conclusion is independent of the chosen sequence.
Since \(q_\mu,q_\nu>d+2p\), both \(\mu\) and \(\nu\) have finite \(2p\)-moments.  The empirical measures converge almost surely in \(W_{2p}\), and coupling each observation with the same Gaussian noise yields
\[
W_{2p}(\mu_m^\sigma,\mu^\sigma)\le W_{2p}(\mu_m,\mu)\to0,\qquad
W_{2p}(\nu_n^\sigma,\nu^\sigma)\le W_{2p}(\nu_n,\nu)\to0
\]
almost surely.

The assumptions of Lemma~\ref{lem:source-potential-stability} are satisfied for the first coordinate.  The displayed \(W_{2p}\)-convergence implies weak convergence of \(\mu_m^\sigma\) to \(\mu^\sigma\) and of \(\nu_n^\sigma\) to \(\nu^\sigma\), and all relevant \(c_p\)-costs are finite since the laws have finite \(p\)-moments.  The limiting source law \(\mu^\sigma\) has a strictly positive \(C^\infty\) density, support \(\R^d\), connected interior, and empty boundary.  In addition, \(h(z)=|z|^p\), \(p>1\), is differentiable and strictly convex, satisfies \(h(z)/|z|\to\infty\), and fulfills the cone condition in Lemma~\ref{lem:source-potential-stability}: for fixed height and aperture, a cone with vertex \(z\) and axis directed toward the origin lies inside \(\{|w|\le |z|\}\) when \(|z|\) is sufficiently large, so the radial function \(|\cdot|^p\) attains its maximum at the vertex.  Applied to \((\mu_m^\sigma,\nu_n^\sigma)\to(\mu^\sigma,\nu^\sigma)\), that lemma furnishes constants \(a_{m,n}\) such that, upon setting
\[
\phi_{m,n}^0=\widetilde\phi_{m,n}-a_{m,n},
\qquad
\psi_{m,n}^0=\widetilde\psi_{m,n}+a_{m,n},
\]
the first potentials \(\phi_{m,n}^0\) converge locally uniformly to a population potential \(\phi_\sigma^0\).  Let \(\psi_\sigma^0\) be its \(c_p\)-conjugate population partner.  The constants in the two coordinates are opposite since the empirical pair is kept \(c_p\)-conjugate throughout.

Along the almost sure event above, the empirical \(p\)-moments of \(\mu_m^\sigma\) and \(\nu_n^\sigma\) are eventually bounded.  Lemma~\ref{lem:growth} gives, before the subtraction of \(a_{m,n}\), a polynomial envelope with an eventually bounded constant.  Since \(\phi_{m,n}^0=\widetilde\phi_{m,n}-a_{m,n}\) converges locally uniformly, evaluation at a fixed point \(x_0\in\R^d\) gives
\[
|a_{m,n}|\le |\widetilde\phi_{m,n}(x_0)|+|\phi_{m,n}^0(x_0)|,
\]
and the two terms on the right are eventually bounded.  Hence, for some finite random constant \(C\),
\[
|\phi_{m,n}^0(x)|+|\psi_{m,n}^0(y)|+|\phi_\sigma^0(x)|+|\psi_\sigma^0(y)|
\le C(1+|x|^p+|y|^p)
\]
eventually.

To identify the second coordinate in the same conjugate normalization, apply Lemma~\ref{lem:source-potential-stability} with the two marginals interchanged.  The assumptions are again satisfied, since \(\nu^\sigma\) also has a smooth density with full support.  The lemma furnishes constants \(c_{m,n}\) such that \(\psi_{m,n}^0-c_{m,n}\) converges locally uniformly to a population second potential.  By Proposition~\ref{prop:unique-potentials}, this limit differs from the \(c_p\)-conjugate partner \(\psi_\sigma^0\) of \(\phi_\sigma^0\) by a constant.  Absorbing that constant into \(c_{m,n}\), we may write the constants as \(d_{m,n}\) and assume
\[
\psi_{m,n}^0-d_{m,n}\to\psi_\sigma^0
\]
locally uniformly.  The envelope just displayed and evaluation at a single fixed point imply that \(d_{m,n}\) is bounded.  Also,
\[
\int\phi_{m,n}^0\,d\mu_m^\sigma
+\int\psi_{m,n}^0\,d\nu_n^\sigma
=W_p(\mu_m^\sigma,\nu_n^\sigma)^p .
\]
The convergence of the corresponding integrals is obtained by truncating to large balls.  Fix \(R<\infty\).  On \(B_R\), local uniform convergence of the potentials and weak convergence of the smoothed empirical laws imply convergence of the truncated integrals.  On \(B_R^c\), the common envelope \(C(1+|x|^p+|y|^p)\) is uniformly integrable along the smoothed empirical laws because they converge in \(W_{2p}\).  Letting \(R\to\infty\), the two dual integrals converge after replacing \(\psi_{m,n}^0\) by \(\psi_{m,n}^0-d_{m,n}\).  With this replacement, the empirical dual value is \(W_p(\mu_m^\sigma,\nu_n^\sigma)^p-d_{m,n}\).  Since \(W_p(\mu_m^\sigma,\nu_n^\sigma)^p\to W_p(\mu^\sigma,\nu^\sigma)^p\), comparison with the limiting dual value forces \(d_{m,n}\to0\).  Hence \(\psi_{m,n}^0\to\psi_\sigma^0\) locally uniformly.

The \(L_2\) convergence follows by applying the same truncation to squared differences.  On each ball the squared differences converge uniformly to zero; outside the ball they are dominated by the uniformly integrable envelope \(C(1+|x|^{2p}+|y|^{2p})\).  The uniform integrability follows from the \(W_{2p}\)-convergence of the smoothed empirical laws, which gives uniformly integrable \(2p\)-moment tails.  Hence
\[
\|\phi_{m,n}^0-\phi_\sigma^0\|_{L_2(\mu^\sigma)}
+\|\psi_{m,n}^0-\psi_\sigma^0\|_{L_2(\nu^\sigma)}
\to0
\]
since \(\mu^\sigma\) and \(\nu^\sigma\) have finite \(2p\)-moments.

We now pass from this stability normalization to the normalization used in the empirical process.
Let \(h_{m,n}^0=\phi_{m,n}^0*\gamma_\sigma\) and \(h_\sigma^0=\phi_\sigma^0*\gamma_\sigma\).  Set
\[
b_{m,n}=\mu h_{m,n}^0=\int\phi_{m,n}^0\,d\mu^\sigma,
\qquad
b=\mu h_\sigma^0=\int\phi_\sigma^0\,d\mu^\sigma .
\]
The \(L_2(\mu^\sigma)\)-convergence just proved yields \(b_{m,n}\to b\).  The normalized pair in the statement is
\[
\widehat\phi_{m,n}=\phi_{m,n}^0-b_{m,n},
\qquad
\widehat\psi_{m,n}=\psi_{m,n}^0+b_{m,n},
\]
while the corresponding population representative is
\[
\phi_\sigma=\phi_\sigma^0-b,
\qquad
\psi_\sigma=\psi_\sigma^0+b .
\]
Hence
\[
\|\widehat\phi_{m,n}-\phi_\sigma\|_{L_2(\mu^\sigma)}
+\|\widehat\psi_{m,n}-\psi_\sigma\|_{L_2(\nu^\sigma)}
\to0 .
\]
Jensen's inequality transfers this to the convolved functions:
\[
\|\widehat h_{m,n}-h_\sigma\|_{L_2(\mu)}^2
\le \|\widehat\phi_{m,n}-\phi_\sigma\|_{L_2(\mu^\sigma)}^2,
\qquad
\|\widehat k_{m,n}-k_\sigma\|_{L_2(\nu)}^2
\le \|\widehat\psi_{m,n}-\psi_\sigma\|_{L_2(\nu^\sigma)}^2 .
\]
This establishes almost sure convergence along the empirical sequence, hence convergence in probability, with the interpretation in outer probability if the selected pair is nonmeasurable.
\end{proof}

The normalized population pair \((h_\sigma,k_\sigma)\) maximizes \(\mathbb M\) over the smoothed feasible dual class.  For empirical measures, let \((\widehat h_{m,n},\widehat k_{m,n})\) be any normalized optimal pair as in Lemma~\ref{lem:l2-potential-stability}.  Probability statements involving this selected pair are read in outer probability unless a measurable selection is fixed; the final expansion concerns the measurable cost statistic itself.

\subsection{Expansion and inference}
\begin{proposition}\label{prop:finite-moment-expansion}
Let \(p>1\), \(\sigma>0\), and \(\mu\ne\nu\).  Suppose \(M_{q_\mu}(\mu)+M_{q_\nu}(\nu)<\infty\) for some \(q_\mu,q_\nu>d+2p\), and let \(m/(m+n)\to\lambda\in(0,1)\).  Then, with \(N_{m,n}=mn/(m+n)\),
\[
T_p^{(\sigma)}(\mu_m,\nu_n)-T_p^{(\sigma)}(\mu,\nu)
=(\mu_m-\mu)\bar h_\sigma+(\nu_n-\nu)\bar k_\sigma
+o_{\PP}(N_{m,n}^{-1/2}).
\]
\end{proposition}

\begin{proof}
The empirical optimal potentials admit polynomial envelope constants controlled by \(M_p(\mu_m^\sigma)+M_p(\nu_n^\sigma)\).  The additional shift imposing \(\mu\widehat h_{m,n}=0\) is bounded by integrating this envelope against the fixed law \(\mu^\sigma\), so it only changes the envelope constant by a deterministic factor on events localized by moments.  Since
\[
M_p(\mu_m^\sigma)\le C_{p,\sigma}\{1+M_p(\mu_m)\},
\qquad
M_p(\nu_n^\sigma)\le C_{p,\sigma}\{1+M_p(\nu_n)\},
\]
and the empirical \(p\)-moments are tight, for every \(\eta>0\) there exists \(L < \infty\) such that, with probability exceeding \(1-\eta\), the population maximizer and the empirical maximizer both lie in \(\mathcal A_L\). By Lemma~\ref{lem:l2-potential-stability},  
\[
\|(\widehat h_{m,n},\widehat k_{m,n})-(h_\sigma,k_\sigma)\|_{L_2(\mu)\oplus L_2(\nu)}\to_{\PP}0,
\]
and the preceding localization satisfies precisely the hypothesis of Lemma~\ref{lem:localized-equicont}.  We obtain
\[
N_{m,n}^{1/2}\Bigl[(\mathbb M_{m,n}-\mathbb M)(\widehat h_{m,n},\widehat k_{m,n})
-(\mathbb M_{m,n}-\mathbb M)(h_\sigma,k_\sigma)\Bigr]
=o_{\PP}(1).
\]
This assertion is interpreted in outer probability if a measurable optimizer has not been fixed; the inequalities below hold for any selected optimal representative, and the final left-hand side is the measurable scalar cost statistic.
Since \((\widehat h_{m,n},\widehat k_{m,n})\) is an empirical maximizer and \((h_\sigma,k_\sigma)\) is the population maximizer,
\[
0\le
\mathbb M_{m,n}(\widehat h_{m,n},\widehat k_{m,n})-
\mathbb M_{m,n}(h_\sigma,k_\sigma)
\]
and
\[
\mathbb M(\widehat h_{m,n},\widehat k_{m,n})-
\mathbb M(h_\sigma,k_\sigma)\le0.
\]
Together with the preceding equicontinuity bound, these inequalities yield
\[
\mathbb M(h_\sigma,k_\sigma)-
\mathbb M(\widehat h_{m,n},\widehat k_{m,n})=o_{\PP}(N_{m,n}^{-1/2}).
\]
The left-hand side is nonnegative and is bounded above by
\[
(\mathbb M_{m,n}-\mathbb M)(\widehat h_{m,n},\widehat k_{m,n})
-(\mathbb M_{m,n}-\mathbb M)(h_\sigma,k_\sigma),
\]
whose size is controlled by the localized equicontinuity just proved.
It remains to write the cost difference in terms of the population potential:
\[
\begin{aligned}
T_p^{(\sigma)}(\mu_m,\nu_n)-T_p^{(\sigma)}(\mu,\nu)
&=
\mathbb M_{m,n}(\widehat h_{m,n},\widehat k_{m,n})-
\mathbb M(h_\sigma,k_\sigma)\\
&=(\mathbb M_{m,n}-\mathbb M)(h_\sigma,k_\sigma)+o_{\PP}(N_{m,n}^{-1/2})\\
&=(\mu_m-\mu)\bar h_\sigma+(\nu_n-\nu)\bar k_\sigma+o_{\PP}(N_{m,n}^{-1/2}),
\end{aligned}
\]
which is the stated expansion.
\end{proof}

The preceding expansion gives the fixed-alternative limit distribution.

\begin{theorem}[Separated CLT]\label{thm:cost-clt}
Let \(p>1\), \(\sigma>0\), and suppose \(M_{q_\mu}(\mu)+M_{q_\nu}(\nu)<\infty\) for some \(q_\mu,q_\nu>d+2p\).  Assume \(\mu\ne\nu\) and \(m/(m+n)\to\lambda\in(0,1)\).  Then
\[
\sqrt{\frac{mn}{m+n}}\left\{T_p^{(\sigma)}(\mu_m,\nu_n)-T_p^{(\sigma)}(\mu,\nu)\right\}
\Rightarrow
\normaldist(0,\tau_{\sigma,p}^2),
\]
where
\[
\tau_{\sigma,p}^2=(1-\lambda)\operatorname{Var}_{\mu}\{\bar h_\sigma(X)\}
+\lambda\operatorname{Var}_{\nu}\{\bar k_\sigma(Y)\}.
\]
The limit is allowed to be degenerate only if the displayed variance is zero.
\end{theorem}

\begin{proof}
The first-order expansion is furnished by Proposition~\ref{prop:finite-moment-expansion}.  The functions \(\bar h_\sigma\) and \(\bar k_\sigma\) have order-\(p\) envelopes, and \(q_\mu,q_\nu>d+2p\) entails in particular \(q_\mu,q_\nu>2p\).  Hence both functions are square integrable.  Independence of the two samples yields
\[
\sqrt{\frac{mn}{m+n}}\,(\mu_m-\mu)\bar h_\sigma
\Rightarrow \normaldist\bigl(0,(1-\lambda)\operatorname{Var}_{\mu}\{\bar h_\sigma(X)\}\bigr)
\]
and
\[
\sqrt{\frac{mn}{m+n}}\,(\nu_n-\nu)\bar k_\sigma
\Rightarrow \normaldist\bigl(0,\lambda\operatorname{Var}_{\nu}\{\bar k_\sigma(Y)\}\bigr).
\]
Summing the independent limits and using the remainder in Proposition~\ref{prop:finite-moment-expansion} proves the result.
\end{proof}

\begin{corollary}\label{cor:distance-clt}
Under the assumptions of Theorem~\ref{thm:cost-clt}, let
\[
\Delta_\sigma=W_p(\mu^\sigma,\nu^\sigma)>0,
\qquad
\widehat D_{m,n}=W_p(\mu_m^\sigma,\nu_n^\sigma).
\]
Then
\[
\sqrt{\frac{mn}{m+n}}\left(\widehat D_{m,n}-\Delta_\sigma\right)
\Rightarrow
\normaldist\left(0,\frac{\tau_{\sigma,p}^2}{p^2\Delta_\sigma^{2p-2}}\right).
\]
\end{corollary}

\begin{proof}
The map \(t\mapsto t^{1/p}\) is differentiable at \(T_p^{(\sigma)}(\mu,\nu)=\Delta_\sigma^p>0\), with derivative \((p\Delta_\sigma^{p-1})^{-1}\).  Applying the delta method to Theorem~\ref{thm:cost-clt} gives the stated variance.
\end{proof}

\begin{lemma}\label{lem:holdout-lln}
Let \(\mu\) be a probability measure, and let \(Z_1,\ldots,Z_N\) be independent observations from \(\mu\), independent of a random function \(f_N\).  Suppose that \(N\to\infty\), that \(f_N\to f\) in \(L_2(\mu)\) in probability for some \(f\in L_2(\mu)\), and that for every \(\eta>0\) there exists \(L<\infty\) and a deterministic envelope \(G\) with \(G^2\in L_1(\mu)\) such that
\[
\PP\{|f_N|\le L G\}\ge 1-\eta
\]
for all sufficiently large \(N\).  Then
\[
N^{-1}\sum_{i=1}^N f_N(Z_i)-\mu f_N\to_{\PP}0,
\qquad
N^{-1}\sum_{i=1}^N f_N(Z_i)^2-\mu f_N^2\to_{\PP}0 .
\]
\end{lemma}

\begin{proof}
It suffices to work on an event of probability at least \(1-\eta\) on which \(|f_N|\le LG\). Conditional on \(f_N\), the first centered average has variance at most \(N^{-1}\mu f_N^2\le N^{-1}L^2\mu G^2\), hence it converges to zero in conditional probability.

For the second centered average, set \(H_N=f_N^2\). On the same event, \(0\le H_N\le L^2G^2\), with \(L^2G^2\in L_1(\mu)\).  For \(B>0\), decompose  \(H_N=H_N^{(B)}+R_N^{(B)}\), where \(H_N^{(B)}=H_N\ind\{L^2G^2\le B\}\). Conditional on \(f_N\), the bounded part has variance at most \(B^2/N\). The remainder satisfies
\[
\mu R_N^{(B)}\le \mu\{L^2G^2\ind(L^2G^2>B)\},
\]
and the empirical average of \(R_N^{(B)}\) has conditional expectation bounded by the same quantity.  Letting first \(N\to\infty\), then \(B\to\infty\), and lastly \(\eta\downarrow0\) proves both assertions.
\end{proof}

\noindent
The variance in Theorem~\ref{thm:cost-clt} depends on the population smoothed dual potentials and is not directly observable.  In practice these potentials are already computed, or approximated, when evaluating the optimal transport statistic.  The estimator below uses one part of the data to learn the potentials and an independent part to estimate their variances.  This sample split avoids an additional empirical-process argument for reusing the same observations in both steps; additive constants in the potentials do not matter because only sample variances are used.  Cross-fitting could be used to reduce the loss of data in applications, but the single-split version is sufficient for consistency.

\begin{proposition}\label{prop:wald}
Under the assumptions of Theorem~\ref{thm:cost-clt}, the asymptotic variance of the cost statistic under the scaling \(\sqrt{mn/(m+n)}\) is
\[
\tau_{\sigma,p}^2=(1-\lambda)\operatorname{Var}_{\mu}\{\bar h_\sigma(X)\}
+\lambda\operatorname{Var}_{\nu}\{\bar k_\sigma(Y)\}.
\]
A consistent estimator is obtained by sample splitting.  Split the two samples into independent training and evaluation parts of sizes \(m_1,m_2\) and \(n_1,n_2\), where \(m_1,m_2,n_1,n_2\to\infty\), \(m=m_1+m_2\), and \(n=n_1+n_2\).  Let \((\widehat h^{\rm tr},\widehat k^{\rm tr})\) be the convolved optimal potentials computed from the training empirical pair \((\mu_{m_1}^{\rm tr},\nu_{n_1}^{\rm tr})\), with any additive normalization.  On the evaluation samples \(\widetilde X_1,\ldots,\widetilde X_{m_2}\) and \(\widetilde Y_1,\ldots,\widetilde Y_{n_2}\), define
\[
\widehat V_\mu^{\rm ho}
=m_2^{-1}\sum_{i=1}^{m_2}
\left\{\widehat h^{\rm tr}(\widetilde X_i)
-m_2^{-1}\sum_{\ell=1}^{m_2}\widehat h^{\rm tr}(\widetilde X_\ell)\right\}^2,
\]
\[
\widehat V_\nu^{\rm ho}
=n_2^{-1}\sum_{j=1}^{n_2}
\left\{\widehat k^{\rm tr}(\widetilde Y_j)
-n_2^{-1}\sum_{\ell=1}^{n_2}\widehat k^{\rm tr}(\widetilde Y_\ell)\right\}^2,
\]
and
\[
\widehat\tau_{\rm ho}^2
=\frac{n}{m+n}\widehat V_\mu^{\rm ho}
+\frac{m}{m+n}\widehat V_\nu^{\rm ho}.
\]
Write \(\widehat\tau_{\rm ho}\) for its nonnegative square root.
Then
\[
\widehat\tau_{\rm ho}^2\to_{\PP}\tau_{\sigma,p}^2 .
\]
If \(\tau_{\sigma,p}^2>0\), an asymptotic \((1-\alpha)\)-level Wald interval for \(T_p^{(\sigma)}(\mu,\nu)\) is
\[
T_p^{(\sigma)}(\mu_m,\nu_n)
\pm z_{1-\alpha/2}\sqrt{\frac{m+n}{mn}}\,\widehat\tau_{\rm ho}.
\]
Under the same nondegeneracy condition, and with \(\widehat D_{m,n}=W_p(\mu_m^\sigma,\nu_n^\sigma)\), an asymptotic Wald interval for \(\Delta_\sigma=W_p(\mu^\sigma,\nu^\sigma)\) at a separated alternative is
\[
\widehat D_{m,n}
\pm z_{1-\alpha/2}\sqrt{\frac{m+n}{mn}}\,
\frac{\widehat\tau_{\rm ho}}{p\widehat D_{m,n}^{p-1}}.
\]
\end{proposition}

\begin{proof}
The variance formula is the variance of the Gaussian limit in Theorem~\ref{thm:cost-clt}.  We next justify the displayed consistency of the estimator based on sample splitting.  Since sample variances are invariant under additive shifts, we may analyze the training potentials after imposing the normalization used in Lemma~\ref{lem:l2-potential-stability}. Applied to the training samples, that lemma yields
\[
\|\widehat h^{\rm tr}-h_\sigma\|_{L_2(\mu)}
+\|\widehat k^{\rm tr}-k_\sigma\|_{L_2(\nu)}
\to_{\PP}0 .
\]
The polynomial envelopes are of order \(p\), and the condition \(q_\mu,q_\nu>d+2p\) entails finite \(2p\)-moments.  Hence the Cauchy--Schwarz inequality gives
\[
\mu[(\widehat h^{\rm tr})^2]\to_{\PP}\mu h_\sigma^2,\qquad
\mu\widehat h^{\rm tr}\to_{\PP}\mu h_\sigma,
\]
and similarly on the \(\nu\)-side.  The population variances of the training potentials converge in probability to
\(\operatorname{Var}_{\mu}\{\bar h_\sigma(X)\}\) and
\(\operatorname{Var}_{\nu}\{\bar k_\sigma(Y)\}\).

It remains to replace these population variances by variances computed from the evaluation samples.  Conditional on the training samples, the evaluation observations are independent.  On events whose probabilities tend to one, the training potentials have deterministic envelopes \(L(1+|x|^p)\) and \(L(1+|y|^p)\), respectively. The squared envelopes are integrable under \(\mu\) and \(\nu\) since \(q_\mu,q_\nu>2p\).  Lemma~\ref{lem:holdout-lln} gives
\[
m_2^{-1}\sum_{i=1}^{m_2}\widehat h^{\rm tr}(\widetilde X_i)-\mu\widehat h^{\rm tr}\to_{\PP}0,
\quad
m_2^{-1}\sum_{i=1}^{m_2}\widehat h^{\rm tr}(\widetilde X_i)^2-\mu[(\widehat h^{\rm tr})^2]\to_{\PP}0,
\]
and the analogous two convergences for \(\widehat k^{\rm tr}\) on the \(\nu\)-evaluation sample.  These four convergences yield
\[
\widehat V_\mu^{\rm ho}\to_{\PP}\operatorname{Var}_{\mu}\{\bar h_\sigma(X)\},
\qquad
\widehat V_\nu^{\rm ho}\to_{\PP}\operatorname{Var}_{\nu}\{\bar k_\sigma(Y)\}.
\]
Since \(m/(m+n)\to\lambda\), we have
\(\widehat\tau_{\rm ho}^2\to_{\PP}\tau_{\sigma,p}^2\).

Slutsky's theorem gives the cost interval.  The distance interval follows by combining Slutsky's theorem with the factor \((p\Delta_\sigma^{p-1})^{-1}\) from the delta method and replacing \(\Delta_\sigma\) by the consistent estimator \(\widehat D_{m,n}\).
\end{proof}

\begin{remark}[Null case]\label{rem:null-limit}
At the null \(\mu=\nu\), the separated linear expansion above degenerates.  The relevant limit with fixed smoothing is generally a non-Gaussian functional of the full smoothed empirical process, as described in \citep{GoldfeldKatoNietertRioux2024} under stronger assumptions on the smooth empirical process.  The present separated CLT does not provide a null limit theorem.
\end{remark}

\section{Extensions and limitations}
\label{sec:extensions}

The preceding sections focus on Gaussian smoothing and the power cost \(|x-y|^p\).  This is the setting in which both parts of the paper are available: the rate proof uses derivative bounds for smoothed potentials, while the CLT also uses injectivity of Gaussian convolution, full support of the smoothed laws, and uniqueness and stability of potentials.  The rate argument is less rigid than the CLT argument.  The following two propositions isolate the assumptions that are actually used for rates.  They are meant as transfer principles rather than as a separate theory of general costs.

\begin{proposition}[Fixed smoothing kernels]\label{prop:fixed-kernel-extension}
Let \(K_\sigma\) be a fixed probability law on \(\R^d\) with density \(\kappa_\sigma\).  Suppose that, for some integer \(s>d/2\),
\[
\int_{\R^d}(1+|z|^p)|D^\alpha \kappa_\sigma(z)|\,dz<\infty,
\qquad |\alpha|\le s .
\]
Define
\[
T_{p,K}^{(\sigma)}(\mu,\nu)=W_p(\mu*K_\sigma,\nu*K_\sigma)^p .
\]
Then the rate and expectation conclusions of Theorems~\ref{thm:main-prob} and~\ref{thm:mean} continue to hold for \(T_{p,K}^{(\sigma)}\), with constants depending also on \(K_\sigma\).  In particular, if \(M_{q_\mu}(\mu)<\infty\), \(M_{q_\nu}(\nu)<\infty\), and \(q_\mu,q_\nu>p\), then
\[
\bigl|T_{p,K}^{(\sigma)}(\mu_m,\nu_n)-T_{p,K}^{(\sigma)}(\mu,\nu)\bigr|
=O_{\PP}\{\rho_{q_\mu,p,d}(m)+\rho_{q_\nu,p,d}(n)\}.
\]
If \(q_\mu,q_\nu\ge2p\), the corresponding expectation bound also holds.  If \(W_p(\mu*K_\sigma,\nu*K_\sigma)>0\), this order transfers to the smoothed distance \(W_p(\mu_m*K_\sigma,\nu_n*K_\sigma)\).
\end{proposition}

\begin{proof}
Only the smoothing estimate changes.  Lemma~\ref{lem:twosample-dual} and Lemma~\ref{lem:growth} are unchanged for the power cost, and the derivative condition gives the analogue of Lemma~\ref{lem:gaussian-smooth} up to order \(s>d/2\): if \(|f(y)|\le L(1+|y|^p)\), then
\[
\max_{|\alpha|\le s}|D^\alpha(f*\kappa_\sigma)(x)|
\le C L(1+|x|^p).
\]
The convolved dual potentials are contained in the shell classes \(\cH_L\), with constants depending on \(K_\sigma\).  The shell bracketing bound, localization by empirical \(p\)-moments, and the expectation argument from Theorems~\ref{thm:main-prob} and~\ref{thm:mean} then apply.  The assertion for separated distances follows from the local Lipschitz argument used in Corollaries~\ref{cor:separated-distance} and~\ref{cor:mean-separated-distance}.
\end{proof}

The separation condition in Proposition~\ref{prop:fixed-kernel-extension} should not be replaced by \(\mu\ne\nu\) for an arbitrary kernel.  For the Gaussian kernel this replacement is valid since the Gaussian characteristic function has no zeros, so convolution by \(\gamma_\sigma\) is injective.  A general kernel may identify distinct laws after convolution, and testing interpretations based on the smoothed distance must account for this possible loss of injectivity.

The next statement gives the corresponding observation for the cost.  The essential requirement is not the exact power form, but the availability of order-\(p\) envelopes for optimal dual potentials after normalization.  This condition is substantive and must be verified for a chosen loss; the proposition records what follows once it is available.

\begin{proposition}[Translation costs with polynomial growth]\label{prop:loss-extension}
Let \(c(x,y)=\ell(x-y)\), where \(\ell:\R^d\to[0,\infty)\) is lower semicontinuous, convex, and satisfies \(\ell(z)\le C(1+|z|^p)\).  Let \(K_\sigma\) satisfy the derivative condition in Proposition~\ref{prop:fixed-kernel-extension}.  Assume, in addition, that for all population and empirical smoothed pairs considered below, the associated optimal dual potentials admit representatives with a common order-\(p\) polynomial envelope of the form
\[
|\varphi(x)|\le C(1+|x|^p),\qquad
|\psi(y)|\le C(1+|y|^p)
\]
after the usual normalization, with constants controlled by the relevant \(p\)-moments.  Then the two-sample finite-moment rate and the expectation bound remain valid for the smoothed optimal transport cost
\[
T_{\ell,K}^{(\sigma)}(\mu,\nu)=\inf_{\pi\in\Pi(\mu*K_\sigma,\nu*K_\sigma)}
\int_{\R^d\times\R^d}\ell(x-y)\,\pi(dx,dy).
\]
The rate is again \(\rho_{q_\mu,p,d}(m)+\rho_{q_\nu,p,d}(n)\) in probability for \(q_\mu,q_\nu>p\), and in expectation under the corresponding square-integrability condition.
\end{proposition}

\begin{proof}
The assumed polynomial envelopes replace Lemma~\ref{lem:growth}, and the derivative condition on \(K_\sigma\) gives the shell-class inclusion used in Proposition~\ref{prop:fixed-kernel-extension}.  The dual reduction continues to apply to the cost \(c\), since optimal feasible pairs enter linearly in the two marginals.  Once the two empirical processes are bounded over deterministic shell classes, Lemma~\ref{lem:shell-bound}, localization by empirical moments, and the Cauchy--Schwarz argument used in Theorem~\ref{thm:mean} give the stated bounds.
\end{proof}

The restriction of Proposition~\ref{prop:loss-extension} to rate bounds is deliberate.  The separated central limit theorem is more sensitive to the geometry of the cost and to the structure of optimal potentials.  For the Gaussian kernel and the strictly convex power cost with \(p>1\), uniqueness of normalized potentials and their \(L_2\)-stability yield the scalar expansion in Section~\ref{sec:sep}.  For non-strictly convex losses, or for \(p=1\), one should generally expect directional or set-valued limits rather than a single normal limit.  For other kernels or costs, an analogous CLT would require, at a minimum, uniqueness or an appropriate differentiability theory for the normalized potentials, a summable bracketing integral for the localized smoothed dual class, and stability of the empirical potentials in \(L_2\).

The assumption of a fixed bandwidth is also substantive.  All constants in the derivative bounds depend on the smoothing scale.  If \(\sigma=\sigma_N\to0\), the statistical error has to be balanced against the smoothing bias \(T_p^{(\sigma_N)}(\mu,\nu)-W_p(\mu,\nu)^p\), and the derivative constants typically deteriorate as \(\sigma_N\downarrow0\).  That regime requires a separate bias-variance analysis.  Similarly, at the null the relevant derivative is nonlinear and depends on the full smooth empirical process; bootstrap consistency and inference for heavy tails below the Donsker threshold require additional arguments beyond the separated theory under finite moment assumptions developed here.

\section{Conclusion}
Gaussian smoothing permits a two-sample empirical optimal transport analysis under finite polynomial moments.  The main rate result is obtained at the level of the dual cost, where the two empirical processes enter linearly.  This gives rates in probability for \(q_\mu,q_\nu>p\), expectation bounds under square-integrable envelopes, and distance rates whenever the smoothed population distance is separated from zero.

When the moment order is above the shell threshold, this dual viewpoint also supports fixed-alternative inference.  Uniqueness and \(L_2\)-stability of the smoothed potentials turn the random dual optimizer into a fixed influence function, leading to a two-sample CLT and a consistent sample-splitting variance estimator.  The results should be read in this fixed-bandwidth, separated-alternative regime.  Null inference, \(p=1\), non-strictly convex costs, and vanishing smoothing require different differentiability and bias-variance arguments.

\section*{Data availability statement}

No datasets were generated or analyzed in this theoretical study.


\begin{thebibliography}{99}


\bibitem[Bobkov and Ledoux(2019)]{BobkovLedoux2019}
Bobkov, S., and Ledoux, M. (2019).
\newblock \emph{One-Dimensional Empirical Measures, Order Statistics, and Kantorovich Transport Distances}.
\newblock Memoirs of the American Mathematical Society, 261(1259).

\bibitem[Cosso et~al.(2025)]{CossoMartiniPerelli2025}
Cosso, A., Martini, M., and Perelli, L. (2025).
\newblock Mean convergence rates for Gaussian-smoothed Wasserstein distances and classical Wasserstein distances.
\newblock arXiv:2504.17477.

\bibitem[del Barrio et~al.(2024)]{delBarrioGonzalezSanzLoubes2024}
del Barrio, E., Gonz\'alez-Sanz, A., and Loubes, J.-M. (2024).
\newblock Central limit theorems for general transportation costs.
\newblock \emph{Annales de l'Institut Henri Poincar\'e, Probabilit\'es et Statistiques}, 60(2), 847--873.
\newblock doi:10.1214/22-AIHP1356.

\bibitem[Ding and Niles-Weed(2022)]{DingNilesWeed2022}
Ding, Y., and Niles-Weed, J. (2022).
\newblock Asymptotics of smoothed Wasserstein distances in the small noise regime.
\newblock In \emph{Advances in Neural Information Processing Systems}.

\bibitem[Fang and Santos(2019)]{FangSantos2019}
Fang, Z., and Santos, A. (2019).
\newblock Inference on directionally differentiable functions.
\newblock \emph{The Review of Economic Studies}, 86(1), 377--412.

\bibitem[Fournier and Guillin(2015)]{FournierGuillin2015}
Fournier, N., and Guillin, A. (2015).
\newblock On the rate of convergence in Wasserstein distance of the empirical measure.
\newblock \emph{Probability Theory and Related Fields}, 162, 707--738.

\bibitem[Goldfeld and Greenewald(2020)]{GoldfeldGreenewald2020}
Goldfeld, Z., and Greenewald, K. (2020).
\newblock Gaussian-smoothed optimal transport: metric structure and statistical efficiency.
\newblock In \emph{Proceedings of the Twenty Third International Conference on Artificial Intelligence and Statistics}, PMLR 108, 3327--3337.

\bibitem[Goldfeld et~al.(2020)]{GoldfeldGreenewaldNilesWeedPolyanskiy2020}
Goldfeld, Z., Greenewald, K., Niles-Weed, J., and Polyanskiy, Y. (2020).
\newblock Convergence of smoothed empirical measures with applications to entropy estimation.
\newblock \emph{IEEE Transactions on Information Theory}, 66(7), 4368--4391.

\bibitem[Goldfeld et~al.(2024a)]{GoldfeldKatoNietertRioux2024}
Goldfeld, Z., Kato, K., Nietert, S., and Rioux, G. (2024a).
\newblock Limit distribution theory for smooth \(p\)-Wasserstein distances.
\newblock \emph{The Annals of Applied Probability}, 34(2), 2447--2511.

\bibitem[Goldfeld et~al.(2024b)]{GoldfeldKatoRiouxSadhu2024}
Goldfeld, Z., Kato, K., Rioux, G., and Sadhu, R. (2024b).
\newblock Statistical inference with regularized optimal transport.
\newblock \emph{Information and Inference: A Journal of the IMA}, 13(1), iaad056.

\bibitem[Manole and Niles-Weed(2024)]{ManoleNilesWeed2024}
Manole, T., and Niles-Weed, J. (2024).
\newblock Sharp convergence rates for empirical optimal transport with smooth costs.
\newblock \emph{The Annals of Applied Probability}, 34(1B), 1108--1135.

\bibitem[Nietert et~al.(2021)]{NietertGoldfeldKato2021}
Nietert, S., Goldfeld, Z., and Kato, K. (2021).
\newblock Smooth \(p\)-Wasserstein distance: structure, empirical approximation, and statistical applications.
\newblock In \emph{Proceedings of the 38th International Conference on Machine Learning}, PMLR 139, 8172--8183.

\bibitem[Sadhu et~al.(2021)]{SadhuGoldfeldKato2021}
Sadhu, R., Goldfeld, Z., and Kato, K. (2021).
\newblock Limit distribution theory for the smooth 1-Wasserstein distance with applications.
\newblock arXiv:2107.13494.

\bibitem[Santambrogio(2015)]{Santambrogio2015}
Santambrogio, F. (2015).
\newblock \emph{Optimal Transport for Applied Mathematicians}.
\newblock Birkh\"auser.

\bibitem[Sommerfeld and Munk(2018)]{SommerfeldMunk2018}
Sommerfeld, M., and Munk, A. (2018).
\newblock Inference for empirical Wasserstein distances on finite spaces.
\newblock \emph{Journal of the Royal Statistical Society: Series B}, 80(1), 219--238.

\bibitem[Staudt et~al.(2025)]{StaudtHundrieserMunk2025}
Staudt, T., Hundrieser, S., and Munk, A. (2025).
\newblock On the uniqueness of Kantorovich potentials.
\newblock \emph{SIAM Journal on Mathematical Analysis}, 57(2), 1452--1482.
\newblock doi:10.1137/24M1658966.

\bibitem[van der Vaart and Wellner(1996)]{vdVW1996}
van der Vaart, A. W., and Wellner, J. A. (1996).
\newblock \emph{Weak Convergence and Empirical Processes}.
\newblock Springer.

\bibitem[Villani(2009)]{Villani2009}
Villani, C. (2009).
\newblock \emph{Optimal Transport: Old and New}.
\newblock Springer.

\bibitem[Weed and Bach(2019)]{WeedBach2019}
Weed, J., and Bach, F. (2019).
\newblock Sharp asymptotic and finite-sample rates of convergence of empirical measures in Wasserstein distance.
\newblock \emph{Bernoulli}, 25(4A), 2620--2648.

\bibitem[Zhang et~al.(2021)]{ZhangChengReeves2021}
Zhang, Y., Cheng, X., and Reeves, G. (2021).
\newblock Convergence of Gaussian-smoothed optimal transport distance with sub-gamma distributions and dependent samples.
\newblock In \emph{Proceedings of the Twenty Fourth International Conference on Artificial Intelligence and Statistics}, PMLR 130, 2422--2430.

\end{thebibliography}
\end{document}